\documentclass[a4paper, 11pt]{amsart}

\usepackage{amsfonts}
\usepackage{amsmath}
\usepackage{amsthm}
\usepackage[pdftex]{graphicx}
\usepackage{hyperref}

\usepackage{amssymb,amscd,enumerate}

\usepackage{amssymb,amscd,enumerate}

\newtheorem{theorem}{Theorem}[section]
\newtheorem{definition}[theorem]{Definition}

\newtheorem{lemma}[theorem]{Lemma}
\newtheorem{corollary}[theorem]{Corollary}
\newtheorem{remark}[theorem]{Remark}
\newtheorem{example}[theorem]{Example}

\newenvironment{conj1}{\textsc{Conjecture} 1.}{}
\newenvironment{conj2}{\textsc{Conjecture} 2.}{}

\newenvironment{conj2bir}{\textsc{Conjecture} 2b.}{}

\input xy
\xyoption{all}

\newcommand{\C}{\mathbb{C}}
\newcommand{\B}{\mathbf{B}}
\newcommand{\R}{\mathbb{R}}

\newcommand{\N}{\mathbb{N}}
\newcommand{\Q}{\mathbb{Q}}
\newcommand{\D}{\Delta}
\newcommand{\la}{\lambda}
\newcommand{\supp}{\mathrm{Supp}}
\newcommand{\cosupp}{\mathrm{Cosupp}}
\newcommand{\exc}{\mathrm{Exc}}

\newcommand{\G}{\Gamma}
\newcommand{\Nklt}{\mathrm{Nklt}}
\newcommand{\Dc}{\widehat{\D}}
\newcommand{\ord}{\mathrm{ord}}
\newcommand{\Div}{\mathrm{Div}}
\renewcommand{\dim}{\mathrm{dim}}
\newcommand{\codim}{\mathrm{codim}}
\newcommand{\Null}{\mathrm{Null}}
\newcommand{\NSNC}{\mathrm{NSNC}}
\newcommand{\Sing}{\mathrm{Sing}}
\newcommand{\CLC}{\mathrm{CLC}}

\newcommand{\Ndlt}{\mathrm{Ndlt}}

\newcommand{\A}{\mathbf{A}}
\newcommand{\bL}{\mathbf{L}}

\begin{document}
\setlength{\parindent}{0pt}

\title[On the semiampleness of the positive part of CKM Zariski dec.]{On the semiampleness of the positive part of CKM Zariski decompositions}
\date{\today}

\author{Salvatore Cacciola}

\address{Dipartimento di Matematica e Fisica, Universit\`a di Roma Tre, Largo San Leonardo Murialdo, 1, 00146, Roma, Italy}
\email{cacciola@mat.uniroma3.it}

\begin{abstract}
We study graded rings associated to big divisors on LC pairs whose difference with the log-canonical
divisor is nef.\\
For divisors that are positive enough at the LC centers of the pair, we prove the finite generation of such rings if the pair is DLT or the dimension is low, given that a Zariski decomposition exists.

\end{abstract}
\maketitle

\section{Introduction}
Given a normal complex projective variety $X$,
the graded ring associated to a Cartier divisor $D$ on $X$ is
$$R(X,D):=\bigoplus_{m \in \N} H^0(X,\mathcal{O}_X(mD)).$$
Also, we can naturally extend this definition to $\Q$-Cartier divisors, by considering integral parts.

This ring might be not finitely generated as a $\C$-algebra as soon as the dimension of $X$ is at least $2$.
On the other hand the finite generation of the (log)-canonical ring $R(X,K_X+\D)$ is one of the main conjectures of the Minimal Model Program (see for example \cite{Mor87} and \cite{Kaw09})
and the proof of this result for KLT pairs is one of the most important results of the outstanding paper \cite{BCHM10}.

We are interested in studying the finite generation of graded rings associated to big divisors whose difference with the (log)-canonical divisor is nef.
By \cite{BCHM10} we easily get the following:
\begin{theorem}\label{KLT-basic}
Let $X$ be a normal projective variety and let $\D$ be an effective Weil $\Q$-divisor on $X$ such that  $(X,\D)$ is a Kawamata log terminal (KLT) pair.
If $D$ is a $\Q$-Cartier divisor on $X$ such that
\begin{description}
\item[A] $D$ is big;
\item[B]$aD-(K_X+\D)$ is nef for some rational number $a\geq 0$;

\end{description}
 then $R(X,D)$ is finitely generated.
\end{theorem}


In this paper we want to generalize this result to the case when the pair $(X,\D)$ is log canonical (LC).
Note that, in general, LC pairs are much more difficult to treat than KLT pairs. One typical reason is that we have much less freedom to perturb the boundary divisor $\D$ without worsen the singularities of the pair.

In particular Theorem \ref {KLT-basic} is no longer true in general for log canonical pairs, as shown in section \ref{examples}, so that we need to add an hypothesis
that ensures that $D$ is positive enough at the LC centers of $(X,\D)$, the subvarieties of $X$ where this pair fails to be KLT.

More precisely we formulate the following conjecture:

\vspace{2mm}

\begin{conj1}
\begin{em}
Let $(X,\D)$ be an LC pair, with $\D$ effective.

If $D$ is a $\Q$-Cartier divisor on $X$ which satisfies \textbf{\emph{A}} and \textbf{\emph{B}}
and the augmented base locus $\B_+(D)$ (see Definition \ref{augmented_base_locus}) does not contain any LC center of the pair $(X,\D)$, then the graded ring $R(X,D)$ is finitely generated.
\footnote{Note that the same conjecture was formulated and proved independently by C. Birkar and Z. Hu (see \cite[Corollary 1.3]{BH13}).}
\end{em}
\end{conj1}

\vspace{2mm}

Note that, as shown in Theorem \ref{DLT_BCHM}, in the case when $(X,\D)$ is DLT Conjecture 1 again follows from \cite{BCHM10}.

In Section 3 we easily prove the following:
\begin{theorem}[see Theorem \ref{antonella}]
Conjecture 1 holds in dimension $n$ if we assume the existence of minimal models for $\Q$-factorial DLT pairs of log-general type of dimension $n$
and the abundance conjecture for semi log canonical pairs of dimension $n-1$.
\end{theorem}

As a corollary we get that Conjecture 1 holds if $\dim\;X\leq 4$ (see Corollary \ref{dim4}).

A very useful tool, that can be used when trying to investigate the finite generation of graded rings $R(X,D)$, is the Zariski decomposition.

We say that a $\Q$-Cartier divisor $D$ on $X$ admits a $\Q$-Zariski decomposition
in the sense of Cutkosky-Kawamata-Moriwaki (or a $\Q$-CKM Zariski decomposition) $D=P+N$ if
\begin{itemize}
\item $P$ and $N$ are $\Q$-Cartier divisors;
\item $P$ is nef and $N$ is effective;
\item There exists an integer $k>0$ such that $kD$ and $kP$ are integral divisors and
for every $m \in \N$ the natural map
$$H^0(X,\mathcal{O}_X(kmP))\rightarrow H^0(X,\mathcal{O}_X(kmD))$$ is bijective.
\end{itemize}

Every pseudoeffective divisor on a smooth projective surface admits a Zariski decomposition (see \cite{Fuj79}).
On the other hand
 in higher dimension there exist big divisors such that no birational pullbacks
admit a Zariski decomposition
even if we allow $P$ and $N$ to be $\R$-divisors (see\cite{Nak04}).

If, up to birational modifications, there exists a $\Q$-CKM Zariski decomposition of $D$ and  the nef part $P$ of such a decomposition is semiample, then the graded ring $R(X,D)$ is finitely generated.

Hence a classical approach to prove the finite generation of $R(X,D)$ (see for example \cite{Kaw87}) is to split the proof in the following two steps:
\begin{enumerate}
\item There exists a birational morphism $f:Z\to X$ such that $f^*(D)=P+N$ is a $\Q$-CKM Zariski decomposition;
\item $P$ is semiample.
\end{enumerate}


In this paper, from Section 4 onwards, we try to prove that the divisors we are considering have a finitely generated graded ring, given that condition (i) holds.

Moreover, in this case, we can lighten the hypothesis on the $\B_+$
by using the notion of logbig divisors, introduced by Miles Reid.

A big $\Q$-Cartier divisor is logbig for an LC pair if its restriction to every LC center of the pair is still big (see definition \ref{logbig_def}).

Note that if $D=P+N$ is a Zariski decomposition, a sufficient condition for $P$, the positive part of the decomposition, to be logbig for the pair $(X,\D)$ is that the augmented base locus $\B_+(D)$ does not contain any LC center of the pair.


Hence, up to the existence of a Zariski decomposition, the following conjecture generalizes Conjecture 1:

\vspace{2mm}

\begin{conj2}
Let $(X,\D)$ be an LC pair, with $\D$ effective.

Suppose that $D$ is a $\Q$-Cartier divisor on $X$ satisfying \textbf{\emph{A}} and \textbf{\emph{B}},
there exists a $\Q$-CKM Zariski decomposition $D=P+N$
and $P$ is logbig for the pair $(X,\D)$. Then $P$ is semiample, so that $R(X,D)$ is finitely generated.
\end{conj2}

\vspace{2mm}

In the case $N=0$, a very similar result was stated by Miles Reid in \cite{Rei93}
and was proved by Florin Ambro in the more general setting of quasi-log varieties in \cite[Theorem 7.2]{Amb01} (see also \cite[Theorem 4.4]{Fuj09}).
Moreover, when $X$ is smooth and $\D$ and $N$ are divisors with simple normal crossing support, Conjecture 2 is true and follows from \cite[Theorem 5.1]{Fuj12}.
In fact the saturation condition discussed by Fujino corresponds to the properties of the positive part of the Zariski decomposition.
By using Fujino's theorem, in Section 4 we prove the following:

\begin{theorem}[see Theorem \ref{DLT_logbig}]
Conjecture 2 holds if $(X,\D)$ is a DLT pair.
\end{theorem}

\vspace{2mm}







On the other hand, note that for a divisor $D$ the existence of a Zariski decomposition  is a very strong property in general,
while it is more likely that a birational pullback of $D$ admits one.




Hence it is natural to generalize Conjecture 2 as follows (the ``b" stands for ``birational"):

\vspace{2mm}

\begin{conj2bir}
\begin{em}
Let $(X,\D)$ be an LC pair, with $\D$ effective.

Let $D$ be a $\Q$-Cartier divisor on $X$ which satisfies \textbf{\emph{A}}and \textbf{\emph{B}}.
Suppose there exists a birational morphism $f:Z\to X$ and a $\Q$-divisor $\D_Z$ on $Z$ such that
$K_Z+\D_Z=f^*(K_X+\D)$, $f^*(D)=P+N$ is a $\Q$-CKM Zariski decomposition and $P$ is logbig for the pair $(Z,\D_Z)$.

Then $P$ is semiample, so that $R(X,D)$ is finitely generated.
\end{em}
\end{conj2bir}



\vspace{2mm}



\vspace{2mm}

In section 5 we give sufficient conditions on the Zariski-decomposed divisor and on the geometry of the LC centers of the pair in order
to have the semiampleness of the positive part (see  \ref{overDLT} and \ref{Nklt2}). As a corollary we prove the following:

\begin{theorem}[see Corollary \ref{dim3}]
Conjecture 2b holds if $\dim\; X\leq 3$.
\end{theorem}

Note that in some of our results we can replace conditions \textbf{A} and \textbf{B} with the more usual hypothesis
\begin{itemize}
\item $aD-(K_X+\D)$ big and nef for some $a \in \Q^+$;
\end{itemize}

as shown in section 7.

\subsection*{Acknowledgments} I am extremely grateful to my supervisor, Prof. Angelo Felice Lopez,
for his patient guidance and his constant encouragement.
I am also indebted to Prof. Paolo Cascini for suggesting Theorem \ref{allianz} and to Prof. S\'ebastien Boucksom for many valuable comments about a previous draft of the paper.
Moreover I would like to thank Dr. Lorenzo Di Biagio for many helpful discussions and the anonymous referee for his/her precious suggestions.

\section{Preliminaries}

\subsection{Notation and conventions}
We will always work over the field of complex numbers $\mathbb{C}$.
Given a (complex) normal projective variety $X$, we denote by $\Div(X)$ the set of Cartier divisors on $X$
and by $\Div_\Q(X)$ the set of divisors such that an integral multiple is Cartier.

If $\mu:Y \to X$ is a birational morphism, we denote by $\exc(\mu)$ the
exceptional locus of $\mu$, that is the complement of the biggest open subset of $Y$ on which $\mu$ is an isomorphism.

A \textit{pair} $(X,\D)$ consists of a normal projective variety $X$ and a Weil $\mathbb{Q}$-divisor $\D$ on $X$ such that
$K_X+\D \in \Div_\Q(X)$.

We say that a subvariety $V\subseteq X$ is a \textit{log canonical center} or a \textit{LC center} of the pair $(X,\D)$ if
it is the image, through a  proper birational morphism, of a divisor $E$ over $X$ such that the discrepancy $a(E,X,\D)\leq -1$.

We define $\CLC(X,\D)$:=\{LC centers of the pair $(X,\D)$\}.

Moreover the \emph{non-klt locus} of the pair $(X,\D)$ is
$$\Nklt(X,\D):=\bigcup_{V\in \CLC(X,\D)} V  $$


$ $

We refer to \cite{KM00} for the standard definitions about singularities of pairs that we do not give explicitly.







\begin{definition}\begin{em}
Let $X$ be a normal variety and let $D$ be a Weil $\R$-divisor on $X$.
If we write $D=\sum d_i D_i$, where the $D_i$ are distinct prime divisors, we define

$$D^{\geq1}:=\sum_{d_i\geq 1} d_i D_i,\quad \qquad D^{=1}:=\sum_{d_i=1} D_i.$$
\end{em}\end{definition}

\subsection{Augmented base locus}

\begin{definition}\begin{em}\label{augmented_base_locus}(cf. \cite[Def.\ 1.2]{ELMNP06}).
Let $X$ be a normal projective variety, let $D\in \Div_\Q(X)$. The \textit{augmented base locus} of $D$ is
$$\B_+(D):= \bigcap_{\substack{E\in \Div_\Q(X), E \geq 0\\ D-E \text{ ample }}} \supp(E),$$
if $D$ is big; otherwise $\B_+(D):=X$ by convention.
\end{em}\end{definition}

\begin{definition}\begin{em}(cf. \cite[Def.\ 10.3.4]{Laz04}).
Let $X$ be a normal projective variety and take $L\in \Div_{\Q}(X)$ big and nef.
The \emph{null locus} $\Null(L)\subseteq X$ of $L$ is the union of all positive dimensional subvarieties $V\subseteq X$ such that
$$(L^{\dim \; V} \cdot V)=0.$$
This is a proper algebraic subset of $X$ by \cite[Lemma 10.3.6]{Laz04}.
\end{em}\end{definition}



We will use the following lemma:

\begin{lemma}\label{B+divisori}
Let $X$ be a normal projective variety and let $P \in \Div_\Q(X)$ be big and nef.

If $E$ is a prime divisor on $X$ such that $P_{|_E}$ is big,
then $\B_+(P)\not\supseteq E$.
\end{lemma}

\begin{proof}
Let $\mu:X'\rightarrow X$ be a resolution of singularities of $X$ and let $\widetilde{E}$ be the strict transform of $E$
through $\mu$.
Then
$$(\mu^*(P)^{n-1}\cdot \widetilde{E})=(P^{n-1}\cdot E)=(P_{|_E})^{n-1}>0,$$
because $P_{|_E}$ is big and nef.

Hence, by \cite[10.3.6]{Laz04}, $\widetilde{E}$ is not an irreducible component of $\Null(\mu^*(P))$.

But $\dim\; \widetilde{E}=\dim\; X'-1\geq \dim\; \Null(\mu^*(P))$, so that $\widetilde{E}$ cannot be strictly contained
in an irreducible component of $\Null(\mu^*(P))$.
Thus $\widetilde{E}\not\subseteq \Null(\mu^*(P))$.

Thanks to Nakamaye's theorem (see \cite[Theorem 10.3.5]{Laz04}) this implies that $\widetilde{E}\not\subseteq \B_+(\mu^*(P))$.

Then, by \cite[Proposition 2.3]{BBP13}, $\widetilde{E}\not\subseteq \mu^{-1}(\B_+(P))$,
so that $E=\mu(\widetilde{E})\not\subseteq \B_+(P)$.
\end{proof}



\begin{lemma}\label{giocoB_+}
Let $(X,\D)$ be an LC pair.
Let $L\in \Div_\Q(X)$ be big and such that $\B_+(L)$ does not contain any LC center of the pair $(X,\D)$.

Then there exist an effective Cartier divisor $\Gamma$ on $X$, not containing any LC center of $(X,\D)$ in its support,
and two positive rational numbers $\la,\mu$, such that
$Bs(|\Gamma|)=\B_+(L)$ and
\begin{enumerate}
\item $L-\mu \Gamma$ is ample;
\item $(X,\D+\la \Gamma)$ is an LC pair;
\item $\CLC(X,\D+\la\Gamma)=\CLC(X,\D)$.
\item $(X,\D+\la\Gamma)$ is DLT if $(X,\D)$ is such.
\end{enumerate}

Moreover $\G$ can be chosen generically in its linear series and, if $L$ is nef, we can choose $\mu=\lambda$.
\end{lemma}

\begin{proof}
By \cite[Prop. 1.5]{ELMNP06} there exists $H$, an ample $\Q$-divisor on $X$, and there exists $m_0 \in \N$ such that
$$\B_+(L)=\B(L-H)=Bs(|m_0(L-H)|).$$

Hence, as $\CLC(X,\D)$ is a finite set,
we can choose a general divisor $\Gamma$ in $|m_0(L-H)|$ such that $\supp(\Gamma)$ does not contain LC
centers of $(X,\D)$.

Define $\mu:=\frac{1}{m_0}$, so that $L-\mu\Gamma\sim_{\Q} H$ is ample.

If $\la>0$ is a sufficiently small rational number, then $\la\leq\mu$,
$(X,\D+\nolinebreak\la \Gamma)$ is an LC pair and $\CLC(X,\D+\la\Gamma)=\CLC(X,\D)$.

Moreover it is easy to see that $(X,\D+\la\Gamma)$ is DLT if $(X,\D)$ is such.

Note also that
$$L-\la\G=(1-\frac{\la}{\mu})L+\frac{\la}{\mu}(L-\mu\G)$$
is ample if $L$ is nef.
\end{proof}

\subsection{Standard log-resolutions}

Let $X$ be a normal projective variety and let $D$ be a reduced Weil divisor on $X$.

A \textit{standard log-resolution} of the pair $(X,D)$ is a log-resolution $f$ of the pair $(X,D)$ such that

\begin{itemize}
\item $f$ is a composition of blowings-up of smooth subvarieties of codimension greater than 1 up to isomorphisms;
\item $f_{|_{f^{-1}(U)}}$ is an isomorphism, where $U=X\setminus (\NSNC(D)\cup \Sing(X))$.

\end{itemize}

If $X$ is smooth and $\mathcal{I}\subseteq \mathcal{O}_X$ is a non zero ideal sheaf, then
a \textit{standard log-resolution} of the ideal sheaf $\mathcal{I}$ is a log-resolution $g$ of $\mathcal{I}$ such that
$g$ is a composition of blowings-up of smooth subvarieties of codimension greater than 1
contained in $\cosupp(\mathcal{I})$
up to isomorphisms.

In particular $g_{|_{g^{-1}(X \setminus \textrm{Cosupp}(\mathcal{I}))}}$ is an isomorphism.

\begin{remark}\label{esistenza_standard}\begin{em}
Given a normal projective variety $X$ and  a reduced Weil divisor $D$ on $X$,
there exists a standard log-resolution of the pair $(X,D)$
(this follows, for example, by \cite[Theorem 4.1.3]{Laz04} and \cite[Theorem 3.5.1]{Fuj07}).

If $Y$ is a smooth projective variety and $\mathcal{I}\subseteq \mathcal{O}_Y$ is a non zero ideal sheaf,
then there exists a standard log-resolution of $\mathcal{I}$
(see for example \cite[Theorem 35]{Kol05}).
\end{em}\end{remark}



\subsection{Zariski decomposition and birational modifications}

For our purposes we need to extend the classical definition of Zariski decomposition
in the sense of Cutkosky-Kawamata-Moriwaki
to some non $\Q$-Cartier cases. From now on we will use the following definition:

\begin{definition}\label{Zariski}\begin{em}
Let $X$ be a normal projective variety and let $D$ be a Weil $\Q$-divisor on $X$.
We say that $D$ admits a $\Q$-Zariski decomposition in the sense of Cutkosky-Kawamata-Moriwaki
(or a $\Q$-CKM Zariski decomposition) $D=P+N$ if
\begin{itemize}
\item $P$ is a $\Q$-Cartier divisor and $N$ is a Weil $\Q$-divisor;
\item $P$ is nef and $N$ is effective;
\item There exists an integer $k>0$ such that $kP$ is Cartier, $kD$ is an integral Weil divisor and for every $m\in \N$ we have an isomorphism
     $$H^0(X,\mathcal{O}_X(kmP))\simeq H^0(X,\mathcal{O}_X(kmD)).$$
\end{itemize}

\end{em}\end{definition}

Note that $kmD$ might not be a Cartier divisor but it still makes sense to consider
the reflexive sheaf $\mathcal{O}_X(kmD)$ and its $H^0$.
In particular if $D$ is $\Q$-Cartier this definition coincides with the one given in the introduction.


\begin{definition}\begin{em}
Let $(X,\D)$ be a pair with $\D$ effective.
We define the b-divisors $\A(\D)$ and $\bL(\D)$:

For every birational morphism $f:Z\to X$,
if $E$ and $F$ are effective Weil $\Q$-divisors on $Z$ without common components such that
$$K_Z+E\equiv f^*(K_X+\D)+F \quad \textrm{and} \quad f_*(E-F)=\D,$$

we put the trace $\A(\D)_Z:=E-F$ and the trace $\bL(\D)_Z:=E$.
\end{em}\end{definition}

The following lemma will be very useful to treat the case when a birational pullback of a given divisor admits a Zariski decomposition.

\begin{lemma}\label{allseason}
Let $(X,\D)$ be a pair such that $\D$ is effective, let $D\in \Div_{\Q}(X)$ and let $a \in \Q$.
If there exists a projective birational morphism $f:Z\rightarrow X$ such that
$f^*(D)=P+N$ is a $\Q$-CKM Zariski decomposition,
then there exist  Weil $\Q$-divisors  $D',P',N',\D_Z$ such that
\begin{itemize}
\item $\D_Z$ is effective;
\item $D'=P'+N'$ is a $\Q$-CKM Zariski decomposition;
\item $\D_Z-N'=\mathbf{A}(\D)_Z-aN$, so that in particular $(Z,\D_Z-N')$ is a pair;
\item $P'=bP$ for some $b>0$;
\item $t_0P'-(K_Z+\D_Z-N')=P+f^*(aD-(K_X+\D))$ for some $t_0 \in\Q$.
\end{itemize}

In particular if
 $D$ is big and $aD-(K_X+\D)$ is nef  \textbf{or} if
 $aD-(K_X+\D)$ is big and nef, then
 $$t_0P'-(K_Z+\D_Z-N')$$
is big and nef.
\end{lemma}

\begin{proof}
Define $a'=-\min\{0,a\}$ and $a''=\max\{0,a\}$, so that $a'\geq 0$, $a''\geq 0$ and
$a=a''-a'$.
Moreover we can write $\A(\D)_Z=A^+ -A^-$,
where $A^+$ and $A^-$ are effective and without common components, so that $A^-$ is $f$-exceptional.
We define
$\D_Z:=A^+ + a'N$, $N':=a''N+A^-$, $P':=(a''+1)P$ and $D':=P'+N'$.

Then it is immediate that $\D_Z$ is effective,
$\D_Z-N'= \mathbf{A}(\D)_Z-aN$ and $P'$ is a positive rational multiple of $P$.

Moreover $P'$ is a nef $\Q$-Cartier divisor, $N'$ is effective
and, by using the hypothesis and Fujita's lemma (see for example \cite[Lemma 1-3-2]{KMM85}), we can see that there exists $k'\in\N$ such that
$$H^0(Z,\mathcal{O}_Z(k'mP'))\simeq H^0(Z,\mathcal{O}_Z(k'mD'))$$
for every $m \in \N$,
%
%
so that $D'=P'+N'$ is a $\Q$-CKM Zariski decomposition.

Now note that
$$P+f^*(aD-(K_X+\D))=(a+1)P+aN-(K_Z+\A(\D)_Z)=$$
$$=(a+1)P+a''N-(K_Z+\A(\D)_Z+a'N)
=(a+1)P-(K_Z+\D_Z-N')=$$
$$=t_0P'-(K_Z+\D_Z-N'),$$

where $t_0=\frac{a+1}{a''+1}\in \Q$.
\end{proof}

\section{Computation via MMP}
In this section we show in Theorem \ref{antonella} that Conjecture 1 holds in dimension $n$ if we assume the existence of minimal models for LC pairs of log-general type of dimension $n$ and the abundance conjecture in dimension $n-1$.
By using that these conjectures hold true in low dimension we obtain Conjecture 1 in dimension less than or equal to 4 (cf. corollary \ref{dim4}).

Before stating the theorems let us fix some notation and definitions that we will use throughout the section:
\begin{itemize}
\item We say that a pair $(X,\D)$ is of \emph{log-general type} if $K_X+\D\in \Div_\Q(X)$ is big;
\item We refer to \cite[definition 3.50]{KM00} for the definition of \emph{minimal model} of a DLT pair.

\item We refer to \cite[Definition 1.1]{Fuj00} for the definition of \emph{semi log canonical} (or \emph{sLC}) $n$-fold:


\item We say that \emph{sLC-abundance} holds in dimension $n$ if for every sLC $n$-fold $(X,\D)$
such that $K_X+\D$ is nef we have that $K_X+\D$ is semiample.
\end{itemize}

\begin{lemma}\label{caldaia}
Let $(X,\D)$ be an LC pair, with $\D$ effective. Let $D\in \Div_\Q(X)$  be such that
\begin{itemize}
\item $D$ is big;
\item $aD-(K_X+\D)$ is nef for some rational number $a\geq 0$;
\item $\B_+(D)$ does not contain any LC center of the pair $(X,\D)$.
\end{itemize}
Then there exists an LC pair $(X,\D')$, with $\D'$ effective, and a rational number $q>0$ such that
$qD-(K_X+\D')$ is ample.

Moreover $\CLC(X,\D)=\CLC(X,\D')$ and $(X,\D')$ is DLT if $(X,\D)$ is such.
\end{lemma}

\begin{proof}
By Lemma \ref{giocoB_+} there exist an effective Cartier divisor $\Gamma$ and two positive rational numbers $\la,\mu$ such that
$D-\mu\Gamma$ is ample and $(X,\D+\la\G)$ is an LC pair such that $\CLC(X,\D)=\CLC(X,\D+\la\G)$.
Moreover $(X,\D+\la\G)$ is DLT if $(X,\D)$ is such.

We put $q:=a+\frac{\la}{\mu}$ and $\D':=\D+\la\G$, so that
$$qD-(K_X+\D')=aD-(K_X+\D)+\frac{\la}{\mu}(D-\mu\G),$$
is ample.
\end{proof}

The next theorem shows that, by using the finite generation of log-canonical rings on KLT pairs (see \cite{BCHM10}), it is easy to get Conjecture 1 for DLT pairs.

\begin{theorem}\label{DLT_BCHM}
Let $(X,\D)$ be a DLT pair and let $D\in Div_\Q(X)$ be such that
\begin{enumerate}
\item $D$ is big;
\item $aD-(K_X+\D)$ is nef for some rational number $a\geq 0$;
\item $\B_+(D)$ does not contain any LC center of the pair $(X,\D)$.
\end{enumerate}

Then $R(X,D)$ is finitely generated
\end{theorem}

\begin{proof}
By lemma \ref{caldaia} we can suppose that $D-(K_X+\D)$ is an ample $\Q$-divisor. Thus, by the following remark, $R(X,D)$ is finitely generated.
\end{proof}

\begin{remark}\begin{em}
As prof. S\'ebastien Boucksom kindly pointed out, if $(X,\D)$ is a DLT pair, then  it follows by \cite[Corollary 1.1.2]{BCHM10} that the graded ring $R(X,K_X+\D+A)$ is finitely generated for every ample divisor $A\in \Div_\Q(X)$.

In fact there exists $\D'\in \Div_{\Q}(X)$ such that $\D'\sim_\Q \D+A$ and $(X,\D')$ is KLT (see for example \cite[Proposition 2.12]{Loh11}).

\end{em}\end{remark}

\begin{theorem}\label{allianz}
Let $(X,\D)$ be an LC pair of log-general type of dimension $n$, with $\D$ effective.

Suppose that minimal models exist for every $\Q$-factorial DLT pair of dimension $n$ of log-general type and that sLC abundance holds in dimension $n-1$.


Then the graded ring $R(X,K_X+\D)$ is finitely generated.
\end{theorem}

\begin{proof}
Up to performing a DLT blow-up (see \cite[Theorem 10.4]{Fuj11}) and passing to a minimal model, we can suppose that $(X,\D)$ is a $\Q$-factorial DLT pair and $K_X+\D$ is nef and big.
Moreover, if we set $S=[\D]$, then $(K_X+\D)_{|_S}$ is semiample by sLC abundance in dimension $n-1$.
Then we conclude by applying, for example, \cite[Theorem 1.1]{Fuj12}.
\end{proof}

\begin{theorem}\label{antonella}
Let $(X,\D)$ be an LC pair of dimension $n$ with $\D$ effective.
Let $D\in \Div_\Q(X)$ be such that

\begin{enumerate}
\item $D$ is big;
\item $aD-(K_X+\D)$ is nef for some rational number $a\geq 0$;
\item $\B_+(D)$ does not contain any LC center of the pair $(X,\D)$.
\end{enumerate}

Also suppose that minimal models exist for every $\Q$-factorial DLT pair of dimension $n$ of log-general type and that sLC-abundance holds
in dimension $n-1$.

Then $R(X,D)$ is finitely generated.
\end{theorem}

\begin{proof}
By Lemma \ref{caldaia} we can suppose that $D-(K_X+\D)$ is ample.





Hence by \cite[Lemma 5.17]{KM00} there exists an effective ample $\Q$-Cartier $\Q$-divisor $H$ such that
$$D-(K_X+\D)\sim_\Q H,$$

and $(X,\D+H)$ is an LC pair.

In other words, if we put $\D_0:=\D+H$, then $D\sim_\Q K_X+\D_0$ and $(X,\D_0)$ is an LC pair.

Therefore we conclude by Theorem \ref{allianz}.
\end{proof}

\begin{corollary}\label{dim4}
Let $(X,\D)$ be an LC pair such that $\D$ is effective and $\dim\; X\leq 4$.
If $D\in \Div_\Q(X)$ is such that

\begin{enumerate}
\item $D$ is big;
\item$aD-(K_X+\D)$ is nef for some rational number $a\geq 0$;
\item $\B_+(D)$ does not contain any LC center of the pair $(X,\D)$;
\end{enumerate}

then $R(X,D)$ is finitely generated.
\end{corollary}

\begin{proof}
sLC-abundance in dimension 3 holds by \cite[Theorem 0.1]{Fuj00}, while
every DLT $\Q$-factorial pair of dimension 4 of log-general type has a minimal model by \cite[Corollary 3.6]{AHK07}.
Hence we can apply Theorem \ref{antonella} and we are done.
\end{proof}

\section{DLT logbig case}

The aim of this section is to prove Theorem \ref{DLT_logbig},
by reducing ourselves to
the hypotheses of \cite[Theorem 5.1]{Fuj12}.
Note that, in particular, Theorem \ref{DLT_logbig} implies Conjecture 2 in the DLT case.

\begin{definition}\label{logbig_def}\begin{em}
Let $(X,\D)$ be a pair and let $L\in \Div_{\Q}(X)$.

We say that $L$ is \emph{logbig} for the pair $(X,\D)$ if $L$
is big and $L_{|_V}$ is big for every $V\in \CLC(X,\D)$.

Moreover given an integer $k\in\{1, \dots, n\}$ we say that $L$ is \emph{logbig in codimension $k$}
if $L$ is big and $L_{|_V}$ is big for every $V\in \CLC(X,\D)$ such that $\codim_X\; V=k$.
\end{em}\end{definition}

\begin{theorem}\label{DLT_logbig}
Let $(X,\D)$ be an LC pair, with $\D$ effective, and let $D \in \Div_\Q(X)$.
If:
\begin{enumerate}
\item $D$ is big;
\item $aD-(K_X+\D)$ is nef for some rational number $a\geq 0$;
\item There exists a projective birational morphism $f:Z\to X$ such that $$f^*(D)=P+N$$ is a $\Q$-CKM Zariski decomposition
and the pair $(Z,\bL(\D)_Z)$ is DLT;
\item $P$ is logbig for the pair $(Z,\bL(\D)_Z)$;
\end{enumerate}
then $P$ is semiample. Hence $R(X,D)$ is finitely generated.
\end{theorem}

\begin{proof}
Let us apply Lemma \ref{allseason} and take $t_0\in \Q$ and $D'$, $P'$, $N'$, $\D_Z$ Weil $\Q$-divisors on $Z$ as in the lemma,
so that in particular $D'=P'+N'$ is a $\Q$-CKM Zariski decomposition.

Define $B:=\D_Z-N'=\A(\D)_Z-aN\leq \bL(\D)_Z$,
so that $t_0P'-(K_Z+B)=P+f^*(aD-(K_X+\D))$ is big and nef.

Moreover $t_0P'-(K_Z+B)$ is logbig for the pair $(Z,\bL(\D)_Z)$ because $P$ is such.

Then we have that $t_0P'-(K_Z+B)$ is logbig for the pair $(Z,B)$, because $\CLC(Z,B)\subseteq \CLC(Z,\bL(\D)_Z)$.

Now, thanks to the main Theorem of \cite{Sza95}, the DLTness of $(Z,\bL(\D)_Z)$ implies that every LC center of $(Z,\bL(\D)_Z)$ is not contained in $\Sing(Z)\cup \NSNC(\bL(\D)_Z)$.

Thus it is easy to see
that every LC center of the pair $(Z,B)$ is not
contained in $\Sing(Z)\cup \NSNC(B)$.

Let $\mu:Z'\rightarrow Z$ be a standard log-resolution of the pair $(Z,B)$,
so that $(Z',\A(B)_{Z'})$ is an LC pair, $Z'$ is smooth and $\A(B)_{Z'}$ has SNCS.

Now we choose $k_0 \in \N$ such that $k_0P'$ is a Cartier divisor, $k_0D'$ is integral and

$$H^0(Z,\mathcal{O}_Z(mk_0P'))\simeq H^0(Z,\mathcal{O}_Z(mk_0D'))$$

for all $m \in \N$.

Moreover if we write $\A(B)_{Z'}=(B')_+-(B')_-$, where $(B')_+$ and $(B')_-$ are effective divisors and they have not common components,
then $\mu_*(\ulcorner (B')_-\urcorner )\leq \ulcorner N'\urcorner $, because $\D_Z$ is effective.
Thus, by the projection formula,
 we get that, for all $m \in \N$,
$$h^0(Z', \mathcal{O}_{Z'}(\mu^*(mk_0P')+ \ulcorner (B')_- \urcorner))\leq h^0(Z,\mathcal{O}_Z(mk_0P'+\ulcorner N'\urcorner))\leq$$
$$\leq h^0(Z,\mathcal{O}_Z(mk_0D'))=h^0(Z,\mathcal{O}_Z(mk_0P'))=h^0(Z',\mathcal{O}_{Z'}(\mu^*(mk_0P'))).$$

Note also that
$$t_0\mu^*(P')-(K_{Z'}+\A(B)_{Z'})\equiv\mu^*(t_0P'-(K_Z+B))$$
is big and nef, being the birational pullback of a big and nef divisor.

$ $

We will prove that $\mu^*(t_0P'-(K_Z+B))$ is logbig for the pair $(Z',\A(B)_{Z'})$:

Let $V \in \CLC(Z',\A(B)_{Z'})$. Then
$\mu(V)\not\subseteq \Sing(Z)\cup \NSNC(B)$.
Thanks to the choice of $\mu$ this implies that $V\not\subseteq \exc(\mu)$,
so that $\mu_{|_V}$ is birational.

Consider the following commutative diagram:

$$\xymatrix{V \ar@{^{(}->}[r] \ar[d]^{\mu_{|_V}} & Z' \ar[d]^{\mu}\\
\mu(V)\ar@{^{(}->}[r]&  Z}$$

We know that $t_0P'-(K_Z+B)$ is logbig for the pair $(Z,B)$,
which implies that $\big{(}t_0P'-(K_Z+B)\big{)}_{|_{\mu(V)}}$ is big.

Then, by birationality of $\mu_{|_V}$, we have that $\mu_{|_V}^*\big{(}(t_0P'-(K_Z+B))_{|_{\mu(V)}}\big{)}$ is a big
$\Q$-divisor on $V$.

But, by commutativity of the diagram, we have that $$\mu_{|_V}^*\big{(}(t_0P'-(K_Z+B))_{|_{\mu(V)}}\big{)}=
\big{(}\mu^*(t_0P'-(K_Z+B))\big{)}_{|_{V}}.$$

Thus we have proved that $\mu^*(t_0P'-(K_Z+B))$ is big when restricted to each LC center
of the pair $(Z',\A(B)_{Z'})$, whence it is logbig for the pair $(Z',\A(B)_{Z'})$.

Hence $t_0\mu^*(P')-(K_{Z'}+\A(B)_{Z'})$ is logbig for the pair $(Z',\A(B)_{Z'})$.

Therefore we can apply \cite[Theorem 5.1]{Fuj12} to the divisor $\mu^*(P')$ and the pair $(Z',\A(B)_{Z'})$, so that $\mu^*(P')$ is semiample,
which implies that $P$ is semiample.
\end{proof}

Let us consider the following alternative version of Theorem \ref{DLT_logbig}:

\begin{corollary}\label{DLT_bpf}
Let $(X,\D)$ be an LC pair with $\D$ effective and let $D \in \Div_\Q(X)$.
If:
\begin{enumerate}
\item $aD-(K_X+\D)$ is big and nef for some rational number $a\geq 0$;
\item There exists a projective birational morphism $f:Z\to X$ such that $$f^*(D)=P+N$$ is a $\Q$-CKM Zariski decomposition
and the pair $(Z,\bL(\D)_Z)$ is a DLT pair;
\item $f^*(aD-(K_X+\D))$ is logbig for the pair $(Z,\bL(\D)_Z)$.
\end{enumerate}
then $P$ is semiample. 
\end{corollary}

\begin{proof}
As in the proof of Theorem \ref{DLT_logbig}, we can apply Lemma \ref{allseason} and take $t_0\in \Q$ and $D'$, $P'$, $N'$, $\D_Z$ Weil $\Q$-divisors on $Z$ as in the lemma.

Define $B:=\D_Z-N'$,
so that $t_0P'-(K_Z+B)=P+f^*(aD-(K_X+\D))$ is nef and logbig for the pair $(Z,\bL(\D)_Z)$, as $f^*(aD-(K_X+\D))$ is such.
The rest of the proof is exactly the same as in Theorem \ref{DLT_logbig}.
\end{proof}

\section{LC logbig case: dimension 3}

\begin{definition}\begin{em}
Let $(X,\D)$ be a pair such that $\dim\; X=n$ and let $k\in\{1,\dots,n\}$. We define

$$\Ndlt(X,\D):=\bigcup_{\substack{V\in \CLC(X,\D)\\V\subseteq \NSNC(\D)\cup \Sing(X)}} V,
\quad\quad\quad \Nklt_k(X,\D):=\bigcup_{\substack{V\in \CLC(X,\D)\\ \dim\; V\leq n-k}} V.$$

Note that if $(X,\D)$ is DLT then $\Ndlt(X,\D)=\emptyset$ by \cite{Sza95}.
\end{em}\end{definition}

\begin{theorem}\label{estate}
Let $(X,\D)$ be a pair and suppose that $\D=\sum_{i \in I} d_i D_i$, where all the $D_i$'s are distinct prime divisors and $d_i\leq 1$ for every $i\in I$.

Moreover suppose that $P \in \Div_{\Q}(X)$ and
we can write $\D=\D_+ - \D_-$, where $\D_+$ and $\D_-$ are effective $\Q$-divisors and the following properties are satisfied:

\begin{enumerate}
\item $P$ is nef;
\item $t_0P-(K_X+\D)$ is ample for some $t_0 \in \Q^+$;
\item There exists $k_0 \in \N$ such that $k_0P$ is a Cartier divisor and  for all $m \in \N$ it holds that
$$H^0(X,\mathcal{O}_X(mk_0P))\simeq H^0(X, \mathcal{O}_X(mk_0P+\ulcorner \D_-\urcorner));$$
\item $\Ndlt(X,\D)=\emptyset$, or $P_{|_{\Ndlt(X,\D)}}$ is semiample;

\item There exists $\mu:X'\rightarrow X$, a standard log-resolution of the pair $(X,\D)$, such that $a(E,X,\D)>-2$ for every prime
divisor $E\subseteq X'$.
\end{enumerate}

Then $P$ is semiample.
\end{theorem}

\begin{proof}
Let $\mu:X'\to X$ be as in the hypothesis.
Note that $\Nklt(X',\A(\D)_{X'})=\supp((\A(\D)_{X'})^{\geq 1})$,
because $X'$ is smooth and $\A(\D)_{X'}$ is SNCS.

Now, by the ampleness of $t_0P-(K_X+\D)$, for all $\mu$-exceptional divisors $E_1,\dots,E_s$ on $X'$ there exist
 arbitrarily small coefficients $\delta_1,\dots,\delta_s \in \Q^+$, such that
$$\mu^*(t_0P-(K_X+\D))-\sum_{j=1}^s \delta_j E_j$$ is ample.
Then, if $0\leq \varepsilon \ll 1$ we have that
$\mu^*(t_0P)-(K_{X'}+(1-\varepsilon)\A(\D)_{X'}+\sum_{j=1}^s \delta_j E_j)$ is still ample.
For every $\varepsilon$ sufficiently small, such that the above condition holds, we define
$$\widehat{\D}_\varepsilon=(1-\varepsilon)\A(\D)_{X'}+\sum \delta_j E_j,$$
so that $\mu^*(mP)-(K_{X'}+\widehat{\D}_\varepsilon)$ is ample for every integer $m \geq t_0$ thanks to the nefness of $P$.
Now we can write
$$\A(\D)_{X'}=\sum_{k \in K} c_k X_k+\sum_{l \in L} a_l Y_l-\sum_{m \in M} b_m Z_m,$$
where, for every $k\in K$ , $l \in L$ and $m \in M$, we have that  $X_k, Y_l, Z_m$ are pairwise distinct prime divisors, and
$$b_m>0 \quad\forall m \in M, \quad 0\leq a_l <1\quad\forall l \in L,\quad 1\leq c_k < 2\quad\forall k \in K:$$
In fact all the coefficients of $\A(\D)_{X'}$ are smaller than 2 because of the choice of $\mu$.

Moreover we can suppose that $\exc(\mu)\subseteq \supp(\sum X_k+\sum Y_l+\sum Z_m)$,
by considering among the $Y_l$'s also the $\mu$-exceptional prime divisors not appearing in $\supp(\A(\D)_{X'})$, with coefficient 0.
Let us define
$$\D'_+:=\sum_{k \in K} c_k X_k +\sum_{l \in L} a_l Y_l; \quad\quad \D'_-:=\sum_{m \in M} b_m Z_m,$$

so that $\D'_+$ and $\D'_-$ are effective, they have no common components and $\A(\D)_{X'}=\D'_+ -\D'_-$.
Moreover for every $k \in K$, $l \in L$ and $m \in M$ we define
$$\gamma_k=\left\{ \begin{array}{ll} \delta_j & \textrm{if } X_k=E_j \\ 0 & \textrm{otherwise}\end{array}\right.;
\quad \gamma_l=\left\{ \begin{array}{ll} \delta_j & \textrm{if } Y_l=E_j \\ 0 & \textrm{otherwise}\end{array}\right. ;
\quad \gamma_m=\left\{ \begin{array}{ll} \delta_j & \textrm{if } Z_m=E_j \\ 0 & \textrm{otherwise}\end{array}\right.$$
so that we can write
$$\widehat{\D}_\varepsilon= \sum_{k \in K} ((1-\varepsilon)c_k+\gamma_k)X_k+ \sum_{l \in L} ((1-\varepsilon)a_l+\gamma_l)Y_l-
\sum_{m \in M} ((1-\varepsilon)b_m-\gamma_m)Z_m.$$
Now we choose $\varepsilon$ and the $\delta_j$'s small enough such that the following inequalities hold:

\begin{itemize}
\item $c'_k:=(1-\varepsilon)c_k+\gamma_k<2 \quad \forall k \in K$;
\item $a'_l:=(1-\varepsilon)a_l+\gamma_l<1 \quad \forall l \in L$;
\item $b'_m:=(1-\varepsilon)b_m-\gamma_m>0 \quad \forall m \in M$,
\end{itemize}
and we define $\widehat{\D}:=\widehat{\D}_\varepsilon$.
Hence $\widehat{\D}=\sum c'_k X_k +\sum a'_l Y_l-\sum b'_m Z_m$,
and
\begin{itemize}
\item $0< c'_k<2 \quad \forall k \in K$;
\item $0\leq a'_l<1 \quad \forall l \in L$;
\item $0<b'_m\leq b_m \quad \forall m \in M$.

\end{itemize}

Note that $$\mu^*(mP)-(K_{X'}+\widehat{\D})$$ is ample for every integer $m \geq t_0$
and $\widehat{\D}$ is a divisor with SNCS because $\supp(\widehat{\D})\subseteq \supp(\A(\D)_{X'})\cup \exc(\mu)$,
so that $\Nklt(X',\widehat{\D})=\supp((\widehat{\D})^{\geq1})$.

Now we define
$$\widehat{\D}_+:=\sum c'_k X_k+\sum a'_l Y_l; \quad \widehat{\D}_-:=\sum b'_m Z_m,$$
so that $\widehat{\D}_+$ and $\widehat{\D}_-$ are effective and $\widehat{\D}=\widehat{\D}_+-\widehat{\D}_-$.

We claim that $\mu_* \ulcorner \widehat{\D}_- \urcorner \leq \ulcorner \D_- \urcorner$:

In fact $\widehat{\D}_-\leq \D'_-$, so that it suffices to show that
$\mu_*\ulcorner \D'_-\urcorner \leq \ulcorner \D_- \urcorner$.
In particular we will show that $\mu_*\D'_-\leq \D_-$.

This holds because, by definition,
$\D'_-=\sum_{a(E,X,\D)>0} a(E,X,\D)E$.
Hence $$\mu_*(\D'_-)=\sum_{a(\mu_*^{-1}D_i,X,\D)>0} a(\mu_*^{-1}D_i,X,\D) D_i=\sum_{d_i<0}-d_i D_i\leq \D_-,$$
as $\D_-$ is effective and $\D_-=\D_+ -\D$, so that, for every $i$,
$\ord_{D_i} \D_-=\ord_{D_i} \D_+ -d_i\geq - d_i.$

$ $

Thanks to the claim, by using the projection formula, we obtain that if $k_0$ is a positive integer as in the hypothesis, then
$$h^0(X',\mathcal{O}_{X'}(\mu^*(k_0mP)+\ulcorner \widehat{\D}_-\urcorner))\leq h^0(X,\mathcal{O}_X(k_0mP+\ulcorner \D_-\urcorner))$$
for all $m\ \in \N$.
But, by hypothesis,
$h^0(X', \mathcal{O}_{X'}(\mu^*(k_0mP)))=h^0(X,\mathcal{O}_X(k_0mP))=h^0(X,\mathcal{O}_X(k_0mP+\ulcorner \D_-\urcorner))$.
Therefore, for all $m\in \N$,
$$H^0(X',\mathcal{O}_{X'}(\mu^*(k_0mP)))\simeq H^0(X', \mathcal{O}_{X'}(\mu^*(k_0mP)+\ulcorner \widehat{\D}_-\urcorner)).$$

We will show the semiampleness of $P$ by applying \cite[Theorem 2.1]{Amb05} to the pair $(X',\widehat{\D})$ and the divisor
$\mu^*(P)$.

In particular, in order to apply the theorem it remains to show that
$\B(\mu^*(P))\cap \Nklt(X',\widehat{\D})=\nolinebreak\emptyset$:

Note that
$$\Nklt(X',\widehat{\D})=\supp(({\widehat{\D})^{\geq 1}})\subseteq \bigcup_{k \in K} X_k =\supp((\A(\D)_{X'})^{\geq 1})=\Nklt(X',\A(\D)_{X'}).$$
Moreover $\Nklt(X',\widehat{\D})\subseteq \exc(\mu)$:

In fact if $k \in K$ is such that $X_k$ is not exceptional,
then $X_k=\mu_*^{-1}G$, for some prime divisor $G$ on $X$. Then $c_k=a(\mu_*^{-1}G,X,\D)=-\textrm{ord}_G \D\geq -1$, thanks to the hypotheses on $\D$.
On the other hand $\gamma_k=0$ because $X_k$ is not exceptional, so that $c'_k=(1-\varepsilon)c_k< c_k\leq 1$.

Thus we get that $\Nklt(X',\widehat{\D})\subseteq \Nklt(X',\A(\D)_{X'})\cap \exc(\mu)$.
Now we define $$T=\sum_{c'_k\geq 1} X_k,$$
so that $T$ is reduced and $T=\supp((\widehat{\D})^{\geq 1})=\Nklt(X',\widehat{\D})$.
In particular $T\subseteq \Nklt(X',\A(\D)_{X'})\cap \exc(\mu)$.

Let $T_0$ be a prime divisor in the support of $T$. Then, on the one hand, $T_0 \subseteq \supp((\A(\D)_{X'})^{\geq 1})$,
that is $a(T_0,X,\D)\leq -1$, which implies that $\mu(T_0) \in \CLC(X,\D)$.

On the other hand $T_0 \subseteq \exc(\mu)$ implies that $\mu(T_0)\subseteq \Sing(X)\cup \NSNC(\D)$,
because $\mu$ is a standard log-resolution of the pair $(X,\D)$.

Hence we get that $\mu(T_0)\subseteq \Ndlt(X,\D)$.
But the same holds for every component of $T$, so that we have
$$\mu(T)\subseteq \Ndlt(X,\D).$$

If $\Ndlt(X,\D)=\emptyset$, then $\mu(T)=\emptyset$, so that $T=\Nklt(X',\widehat{\D})=\emptyset$ and there is nothing to prove.
We can thus assume that $\Ndlt(X,\D)\not=\emptyset$.

Then, as by hypothesis $P_{|_{\Ndlt(X,\D)}}$ is semiample, we get that $P_{|_{\mu(T)}}$ is semiample.

Now we consider the commutative diagram:
$$\xymatrix{T \ar@{^{(}->}[r] \ar[d]^{\mu_{|_T}} & X' \ar[d]^{\mu}\\
\mu(T)\ar@{^{(}->}[r]&  X}$$

As $P_{|_{\mu(T)}}$ is semiample, we have that $\mu_{|_T}^*(P_{|_{\mu(T)}})$ is semiample, which
implies that $\mu^*(P)_{|_T}$ is semiample.


$ $

Now we claim that $\ulcorner -\Dc\urcorner= \ulcorner \Dc_-\urcorner-T$:
In fact
$$\ulcorner -\Dc\urcorner=\sum_{m \in M} \ulcorner b'_m \urcorner Z_m
+\sum_{k \in K} \ulcorner -c'_k\urcorner X_k+\sum_{l \in L} \ulcorner -a'_l \urcorner Y_l.$$
But, for all $l \in L$, we have that $0\geq -a'_l>-1$, so that $\ulcorner -a'_l \urcorner =0$.

Moreover for all $k \in K$, $0>-c'_k>-2$, so that
$$\ulcorner -c'_k\urcorner=\left\{ \begin{array}{ll} -1 & \textrm{if } c'_k\geq 1 \\ 0 & \textrm{if } c'_k<1 \end{array}\right.$$
Thus
$$\ulcorner -\Dc \urcorner=\sum_{m \in M} \ulcorner b'_m \urcorner Z_m-\sum_{c'_k\geq 1} X_k=\ulcorner \Dc_-\urcorner-T.$$
Take $k_1 \in \N$ such that $k_1>t_0$ and $k_1$ is a multiple of $k_0$, so that $k_1P$ is a Cartier divisor and
$$H^0(X', \mathcal{O}_{X'}(\mu^*(k_1mP)))\simeq H^0(X', \mathcal{O}_{X'}(\mu^*(k_1mP)+\ulcorner \widehat{\D}_-\urcorner))$$

for every $m \in \N$.
Let us consider, for every $k \in k_1\N$, the following commutative diagram:
$$\xymatrix{H^0(X', \mathcal{O}_{X'}(\mu^*(kP)+\ulcorner \Dc_-\urcorner)) \ar[r]^{\beta_k} &
H^0(T,\mathcal{O}_T(\mu^*(kP)_{|_{T}}+\ulcorner \Dc_-\urcorner_{|_{T}}))\\
H^0(X',\mathcal{O}_{X'}(\mu^*(kP))) \ar[r]^{\alpha_k} \ar[u]^{\simeq} & H^0(T,\mathcal{O}_T(\mu^*(kP)_{|_T})) \ar[u]^{i_k}
}$$

where the vertical arrow on the left is an isomorphism thanks to the choice of $k_1$.

Note that $i_k$ is injective for every $k \in k_1\N$ because $\ulcorner \Dc_-\urcorner_{|_{T}}$ is effective:

In fact $\ulcorner \Dc_-\urcorner$ is effective
and $\supp(\ulcorner \Dc_-\urcorner)=\supp (\Dc_-)=\cup Z_m$ does not contain any component of $T$.

Let us prove that $\beta_k$ is surjective for every $k\in k_1 \N$.
In particular we prove that $H^1(X',\mathcal{O}_{X'}(\mu^*(kP)+\ulcorner \Dc_-\urcorner-T))=0$:

Note that $\mu^*(kP)-(K_{X'}+\Dc)$ is ample, thanks to the choice of $k_1$,
and $\{\mu^*(kP)-(K_{X'}+\Dc)\}=\{-\Dc\}$ is SNCS.
Then, by Kawamata-Viehweg vanishing theorem (see \cite[Theorem 9.1.20]{Laz04}), we get that
$H^1(X',\mathcal{O}_{X'}(\mu^*(kP)+\ulcorner -\Dc \urcorner))=0$.

But $\ulcorner -\Dc\urcorner= \ulcorner \Dc_-\urcorner-T$. Then $H^1(X',\mathcal{O}_{X'}(\mu^*(kP)+\ulcorner \Dc_- \urcorner-T))=0$,
as required.

$ $

By the commutativity of the diagram, the surjectivity of $\beta_k$ implies that $i_k$ is surjective,
that is $i_k$ is an isomorphism.
Thus $\alpha_k$ is also surjective for every $k \in k_1 \N$.

But $\mu^*(P)_{|_T}$ is semiample, whence there exists $k_2 \in k_1 \N$ such that $\mu^*(k_2P)_{|_T}$ is base point free.

Then the surjectivity of $\alpha_{k_2}$ implies that $Bs(\mu^*(k_2P))\cap T=\emptyset$.
Therefore $\B(\mu^*(P))\cap \Nklt(X',\Dc)=\emptyset$.
\end{proof}

\begin{corollary}\label{overDLT}
Let $(X,\D)$ be a pair such that $\D$ is effective.

Let $D \in \Div_\Q(X)$ be such that
\begin{enumerate}
\item $D$ is big;
\item $aD-(K_X+\D)$ is nef for some $a \in \Q$;
\item There exists a projective birational morphism $f:Z\to X$ such that
$f^*(D)=P+N$ is a $\Q$-CKM Zariski decomposition and
\begin{itemize}
\item $(Z,\A(\D)_Z-aN)$ is an LC pair;
\item $\B_+(f^*(D))$ does not contain any LC center of the pair $(Z,\A(\D)_Z-aN)$;
\item $\Ndlt(Z,\A(\D)_Z-aN)=\emptyset$, or $P_{|_{\Ndlt(Z,\A(\D)_Z-aN)}}$ is semiample.
\end{itemize}
\end{enumerate}

Then $P$ is semiample. 
\end{corollary}

We remark that if $a\geq 0$ the LCness of the pair $(Z,\A(\D)_Z-aN)$ holds
if we suppose that $(X,\D)$ is an LC pair.

\begin{proof}
Let us apply Lemma \ref{allseason} and consider $t_0$, $D' $, $P'$, $N'$, $\D_Z$
as in the lemma, so that $t_0P'-(K_Z+\D_Z-N')$ is big and nef.


Note that $\B_+(P')=\B_+(P)=\B_+(f^*(D))$.
Hence we can apply Lemma \ref{giocoB_+} to the big and nef $\Q$-divisor $P'$ and to the pair $(Z,\D_Z-N')=(Z,\A(\D)_Z-aN)$
 and we find a Cartier divisor $\Gamma$
and a rational number $\la>0$ such that
$P'-\la\Gamma$ is ample, $(Z,\D_Z-N'+\la\Gamma)$ is LC and $\CLC(Z,\D_Z-N'+\la\Gamma)=\CLC(Z,\D_Z-N')$.

Furthermore, we can choose $\Gamma$ generically in its linear series and we have that $Bs(|\Gamma|)=\B_+(P')$.
Then, by Bertini's Theorem, we can suppose that, outside $\B_+(P')$, $\G$ is smooth  and it intersects $\D_Z-N'$ in a simple normal crossing way.

Let us put
$B:=\D_Z-N'+\la\Gamma$.
We will show that the pair $(Z,B)$ and the $\Q$-Cartier divisor $P'$ satisfy the hypotheses of Theorem \ref{estate}.

First of all we have that
$(t_0+1)P'-(K_Z+B)=(P'-\la\G)+ (t_0P'-(K_Z+\D_Z-N'))$
is ample, so that property (ii) holds.

By the LCness of the pair $(Z,B)$ we get that all the coefficients of $B$ are less than or equal to 1 and property (v) holds.
Moreover property (i) is trivially verified and property (iii) follows by the definition of $\Q$-CKM Zariski decomposition
because $\D_Z$ is effective.

$ $

In order to prove that property (iv) holds, it suffices to show that $\Ndlt(Z,B)\subseteq\Ndlt(Z,\D_Z-N')=\Ndlt(Z,\A(\D)_Z-aN)$:

By the choice of $\Gamma$ we have that $\CLC(Z,\D_Z-N')=\CLC(Z,B)$ and
$\NSNC(B)\subseteq \NSNC(\D_Z-N')\cup \B_+(P')$.

Then, if $V\in \CLC(Z,B)$ and $V \subseteq \Sing(Z) \cup \NSNC(B)$,
we get that $V \in \CLC(Z,\D_Z-N')$ and $V \subseteq \Sing(Z)\cup \NSNC(\D_Z-N')\cup \B_+(P')$.
This implies that $V\subseteq \Sing(Z)\cup \NSNC(\D_Z-N')$.
Hence $V \subseteq \Ndlt(Z,\D_Z-N')$, and we get the required inclusion.
Therefore we can apply Theorem \ref{estate}.
\end{proof}

\begin{theorem}\label{estate2}
Let $(X,\D)$ be an LC pair, with $\dim\; X\geq 2$.
Suppose that $P\in \Div_\Q(X)$ and we can write $\D=\D_+-\D_-$, where $\D_+$ and $\D_-$ are effective $\Q$-divisors,
and the following conditions are satisfied:

\begin{enumerate}
\item $P$ is nef;
\item $t_0P-(K_X+\D)$ is nef for some $t_0 \in \Q^+$;
\item There exists $k_0 \in \N$ such that $k_0P$ is a Cartier divisor and for all $m \in \N$ we have
$$H^0(X,\mathcal{O}_X(mk_0P))\simeq H^0(X,\mathcal{O}_X(mk_0P+\ulcorner \D_-\urcorner));$$
\item $\Nklt_2(X,\D)=\emptyset$, or $P_{|_{\Nklt_2(X,\D)}}$ is semiample.
\item $P$ is logbig in codimension 1 for the pair $(X,\D)$, or $t_0P-(K_X+\D)$
is logbig in codimension 1 for the pair $(X,\D)$.
\end{enumerate}

Then $P$ is semiample.
\end{theorem}

\begin{proof}
Let $$L=\left\{ \begin{array}{ll} P & \textrm{if } P \textrm{ is logbig in codimension 1 for the pair }(X,\D) \\ t_0P-(K_X+\D) & \textrm{otherwise}\end{array}\right.$$
Then $L$ is nef and logbig in codimension 1 for the pair $(X,\D)$, so that
$\B_+(L)$ does not contain any divisorial LC center of the pair $(X,\D)$, thanks to Lemma \ref{B+divisori}.

By \cite[Prop. 1.5]{ELMNP06} there exists $H\in \Div_\Q(X)$ ample and there exists $m_0 \in \N$ such that
$$\B_+(L)=\B(L-H)=Bs(|m_0(L-H)|).$$
Hence, we can choose a general divisor $\Gamma$ in $|m_0(L-H)|$ such that $\supp(\Gamma)$ does not contain any divisorial LC
center of $(X,\D)$.
Note that we have
$$L-\la\G \sim_\Q (1-\la m_0)L+\la m_0H$$
is ample if $\la\in (0, \frac{1}{m_0}]$
 because $L$ is nef and $H$ is ample.

Now, for every $\la\in (0, \frac{1}{m_0}]$, let us define $\D_\la=\D+\la\Gamma$.
We will prove that there exists $\la_0 \in \Q^+$
such that, if $\la\in \Q\cap (0,\la_0)$, then $P$ and the pair $(X,\D_\la)$ satisfy the hypotheses of Theorem \ref{estate}.
First of all note that
$$(t_0+1)P-(K_X+\D_\la)=(t_0+1)P-(K_X+\D+\la\G)=P+(t_0P-(K_X+\D))-\la\Gamma=$$
$$=\left\{ \begin{array}{ll} L-\la\G+(t_0P-(K_X+\D)) & \textrm{if } P \textrm{ is logbig in codimension } 1 \\
P+(L-\la\G) & \textrm{otherwise}\end{array}\right.$$
is ample in both cases for every $\la \in (0, \frac{1}{m_0}]$.
Now let us define $$(\D_\la)_+:=\D_++\la\Gamma; \quad (\D_\la)_-:=\D_-.$$
Then $\D_\la=(\D_\la)_+-(\D_\la)_-$, and
$(\D_\la)_+$ and $(\D_\la)_-$ are effective $\Q$-divisors for every $\la>0$, because $\Gamma$ is effective.
Moreover note that, with these definitions, hypotheses (i) and (iii) of Theorem \ref{estate} are trivially verified.

$ $

Now take a rational number $\la'>0$ such that $\supp(\D)\cup\supp(\Gamma)=\supp(\D+\la\G)$ for every $\la\in(0,\la')$.
and let $\mu:X'\rightarrow X$ be a standard log-resolution of the pair $(X,\D+\la\Gamma)$.
For every prime divisor $E\subseteq X'$ we have that
$$a(E,X,\D_\la)=a(E,X,\D+\la\G)=a(E,X,\D)-\la \mbox{ord}_E(\mu^*(\Gamma)),$$
where $a(E,X,\D)\geq -1$, because $(X,\D)$ is an LC pair.

Suppose that $E$ is a prime divisor on $X'$ such that $E$ is not $\mu$-exceptional and $a(E,X,\D)=-1$.

Then $\mu(E)$ is a divisorial LC center of $(X,\D)$, so that
$\mbox{ord}_{\mu(E)} \Gamma=0$, that is $\ord_E (\mu^*(\G))=0$, which implies
$a(E,X,\D_\la)=-1$.

Now define
$$\la_1:=\min_{\substack{ord_E (\mu^*(\G))>0 \\ a(E,X,\D)>-1}}\Big{\{} \frac{1+a(E,X,\D)}{\ord_E(\mu^*(\G))},1\Big{\}}.$$
Then $\la_1>0$ and, if $\la \in \Q \cap (0,\la_1)$, we have that
$a(E,X,\D_\la)>-1$ for every prime divisor $E\subseteq X'$ such that $a(E,X,\D)>-1$.

Define $$\la_2:=\min_{\substack{ord_E (\mu^*(\G))>0 \\ a(E,X,\D)=-1}}\Big{\{} \frac{2+a(E,X,\D)}{\ord_E(\mu^*(\G))},1\Big{\}}.$$

Then $\la_2>0$ and, if $\la \in \Q \cap (0,\la_2)$, we have that
$a(E,X,\D_\la)>-2$ for every prime divisor $E\subseteq X'$ such that $a(E,X,\D)=-1$.

We put $\la_0:=\min\{\la',\la_1,\la_2,\frac{1}{m_0}\}$, so that if $\la \in \Q \cap(0,\la_0)$ then $(X,\D_\la)$ satisfies hypothesis (v) of Theorem \ref{estate}.

Furthermore we can write
$$\D_\la=\sum -a(\mu_*^{-1} B_i,X,\D_\la) B_i,$$

where the $B_i$'s are distinct prime divisors on $X$.
By definition, for every $i$, $\mu_*^{-1} B_i$ is not an exceptional divisors, so that, it follows by the previous calculation
that $-a(\mu_*^{-1} B_i,X,\D_\la)\leq 1$.

$ $


Now let us consider
$$\A(\D)_{X'}=\sum_{E\subseteq X'} -a(E,X,\D) E\mbox{;} \quad \A(\D_\la)_{X'}=\sum_{E\subseteq X'} -a(E,X,\D_\la) E.$$

Thanks to the choice of $\mu$ and $\la$ we have that they both have SNCS.
Let us put $$F:=\sum_{a(E,X,\D_\la)<-1} (-a(E,X,\D_\la)-1)E;$$
$$\widetilde{\D}:=\A(\D_\la)_{X'}-F=\sum_{a(E,X,\D_\la)\geq -1} -a(E,X,\D_\la)E +\sum_{a(E,X,\D_\la)<-1} E.$$

Then we have that $F$ is effective, $\supp(\widetilde{\D})\subseteq \supp(\A(\D_\la)_{X'})$ and all the coefficients of $\widetilde{\D}$ are less than or equal
to 1.
In particular the pair $(X',\widetilde{\D})$ is LC.

Moreover, by the previous calculations, we have that $F$ is exceptional, $\supp(F)\subseteq \supp ((\A(\D)_{X'})^{=1})$
and $\widetilde{\D}^{=1}=(\A(\D)_{X'})^{=1}$.

$ $

Let us show that
$\Nklt_2(X,\D_\la)\subseteq \Nklt_2(X,\D)$:

Let $V$ be an LC center of the pair $(X,\D_\la)$ of codimension greater than one.
Then $V=\mu(W)$ for some $W \in \CLC(X',\A(\D_\la)_{X'})$.

If $W \not \subseteq \supp(F)$,  then $W \in \CLC(X',\widetilde{\D})$,
whence $W$ is an irreducible
component of a finite intersection of prime divisors in the support of $\widetilde{\D}^{=1}=(\A(\D)_{X'})^{=1}$.

Hence $W \in \CLC(X',\A(\D)_{X'})$, which implies that $V=\mu(W)\in \CLC(X,\D)$, so that $V\subseteq \Nklt_2(X,\D)$, because the codimension of $V$ is
greater than 1.

If $W \subseteq \supp(F)$ then there exists a prime divisor $F_0\subseteq \supp(F)$
such that $W\subseteq F_0$.

Then  $F_0 \subseteq \supp(F)\subseteq \supp((\A(\D)_{X'})^{=1})$.
Hence $F_0 \in \CLC(X',\A(\D)_{X'})$, so that $\mu(F_0)\in \CLC(X,\D)$.
Moreover $\codim\; \mu(F_0)\geq  2$, because $F_0$ is exceptional.
Thus
$$V=\mu(W)\subseteq \mu(F_0)\subseteq \Nklt_2(X,\D).$$
This shows that $\Ndlt(X,\D_\la)\subseteq \Nklt_2(X,\D_\la)\subseteq \Nklt_2(X,\D)$, which implies, by the hypotheses,
that $\Ndlt(X,\D_\la)=\emptyset$ or $P_{|_{\Ndlt(X,\D_\la)}}$ is semiample.

Therefore all the hypotheses of Theorem \ref{estate} are satisfied and we get the semiampleness of $P$.
\end{proof}

\begin{corollary}\label{Nklt2}
Let $(X,\D)$ be a pair  with $\D$ effective and $\dim\; X\geq 2$ and let $a \in \Q$.

Let $D \in \Div_\Q(X)$ be such that:
\begin{enumerate}
\item $D$ is big
\item $aD-(K_X+\D)$ is nef;
\item There exists a projective birational morphism $f:Z \to X$ such that
$$f^*(D)=P+N$$ is a $\Q$-CKM Zariski decomposition and
\begin{itemize}
\item $(Z,\A(\D)_Z-aN)$ is an LC pair;
\item $\B_+(f^*(D))$ does not contain divisorial LC centers of the pair $(Z,\A(\D)_Z-aN)$;
\item $\Nklt_2(Z,\A(\D)_Z-aN)=\emptyset$, or $P_{|_{\Nklt_2(Z,\A(\D)_Z-aN)}}$ is semiample;
\end{itemize}

\end{enumerate}

Then $P$ is semiample. 
\end{corollary}

Note that in the case $a\geq 0$ we can just assume that the pair $(X,\D)$
is LC in order to have the LCness of the pair $(Z,\A(\D)_Z-aN)$.

\begin{proof}
%
Thanks to Lemma \ref{allseason}  the corollary follows by applying Theorem \ref{estate2} to the $\Q$-Cartier divisor $P'$ and to the pair
$(Z,\A(\D)_Z-aN)$.
\end{proof}

\begin{corollary}\label{dim3}
Let $(X,\D)$ be a pair such that $\D$ is effective and $\dim\; X\leq 3$,
let $a \in \Q$.
Let $D \in \Div_\Q(X)$ be such that
\begin{enumerate}
\item $D$ is big;
\item $aD-(K_X+\D)$ is nef;
\item There exists a projective birational morphism $f:Z \to X$ such that $f^*(D)=P+N$ is a $\Q$-CKM Zariski decomposition and
\begin{itemize}
\item $(X,\D)$ is an LC pair and $a\geq 0$ (resp. $(Z,\A(\D)_Z-aN)$ is an LC pair);
\item $P$ is logbig for the pair $(Z,\A(\D)_Z)$ (resp. $P$ is logbig for the
 pair $(Z,\A(\D)_Z-aN)$).
\end{itemize}
\end{enumerate}

Then $P$ is semiample. Hence $R(X,D)$ is finitely generated.
\end{corollary}

\begin{proof}
Begin by noting that if $\dim\; X\leq 1$ then the theorem is trivial because every big divisor on a curve is ample.
We can thus assume that $2\leq \dim\; X \leq 3$.

Note also that if $a\geq 0$ and $(X,\D)$ is LC then $(Z,\A(\D)_Z-aN)$ is LC and
$\CLC(Z,\A(\D)_Z-aN)\subseteq \CLC(Z,\A(\D)_Z)$.
Thus we can assume that $P$ is logbig for the LC pair $(Z,\A(\D)_Z-aN)$.

Hence, by Lemma \ref{B+divisori} we get that $\B_+(P)$ does not contain divisorial LC centers of the pair $(Z,\A(\D)_Z-aN)$,
so that the same holds for $\B_+(f^*(D))$.

Then, in order to apply Corollary \ref{Nklt2}, it just remains to show that $P_{|_{\Nklt_2(Z,\A(\D)_Z-aN)}}$ is semiample
if $\Nklt_2(Z,\A(\D)_Z-aN)\not=\emptyset$.


Let $C$ be a connected component of $\Nklt_2(Z,\A(\D)_Z-aN)$. Then, by hypothesis, we have that $0\leq \dim\; C\leq 1$.

If $\dim\; C=0$ then $P_{|_C}$ is trivially semiample.
If $\dim\; C=1$ then we can write $C =\cup_{j=1}^k C_j$, where the $C_j$'s are irreducible curves.

Then we have that $C_j \in \CLC(Z,\A(\D)_Z-aN)$ for every $j\in\{1,\dots,k\}$, so that $P_{|_{C_j}}$ is big, because $P$ is logbig for the pair
$(Z,\A(\D)_Z-aN)$.

But, as $C_j$ is an irreducible curve, this implies that $P_{|_{C_j}}$ is ample for every $j=\{1,\dots,k\}$.
Hence $P_{|_{C}}$ is ample,
so that in particular it is semiample.
\end{proof}

\section{Some consequences of Ambro's theorem}

In this section, as a corollary the main theorem in \cite{Amb05}, we obtain some variants of the results of the previous sections.
Note that we add a strong hypothesis concerning the stable base locus of the positive part of the given Zariski decomposition, but we work on pairs that are not necessarily LC.


\begin{theorem}\label{Ambro}
Let $X$ be a normal projective variety and let $\D$ be an effective Weil $\Q$-divisor.
If $D$ is a Weil $\Q$-divisor such that
\begin{enumerate}
\item There exists a $\Q$-CKM Zariski decomposition
$$D=P+N;$$
\item There exist two rational numbers $a$ and $t_0$, with $a\geq 0$, such that $\Nklt(X,\D-aN)\cap \B(P)=\emptyset$ and
$$t_0P-(K_X+\D-aN)$$ is big and nef,
\end{enumerate}

then $P$ is semiample.
\end{theorem}

\begin{proof}
Let $B=\D-aN$ and let $B_-=aN$.
Then $B+B_-=\D\geq 0$
and $t_0P-(K_X+B)$ is big and nef.

Moreover, by definition of $\Q$-CKM Zariski decomposition, there exists $k_0 \in \N$ such that $k_0>a$, $k_0P$ is a Cartier divisor,
$k_0D$ is integral and

$$H^0(X, \mathcal{O}_X( mk_0P))\simeq H^0(X, \mathcal{O}_X( mk_0D))$$

for all $m \in \N$.
But $\ulcorner B_- \urcorner =\ulcorner aN \urcorner \leq k_0N$.
Hence, for all $m \in \N$, we get that
$$H^0(X, \mathcal{O}_X( mk_0P))\simeq H^0(X, \mathcal{O}_X( mk_0P+\ulcorner B_- \urcorner)).$$

Thus we can apply \cite[Theorem 2.1]{Amb05} and we get the semiampleness of $P$.
\end{proof}

\begin{corollary}\label{ambrocoro}
Let $(X,\D)$ be a pair with $\D$ effective and let $D\in \Div_\Q(X)$.
If
\begin{enumerate}
\item $D$ is big;
\item $aD-(K_X+\D)$ is nef for some $a \in \Q$;
\item There exists a projective birational morphism $f:Z\to X$ such that
$f^*(D)$ admits a $\Q$-CKM Zariski decomposition $$f^*(D)=P+N$$ and
$\Nklt(Z,\A(\D)_Z-aN)\cap \B(P)=\emptyset$;
\end{enumerate}

then $P$ is semiample.
\end{corollary}

\begin{proof}

Let us apply Lemma \ref{allseason},
and consider $t_0\in \Q$ and $D'$, $P'$, $N'$, $\D_Z$ Weil $\Q$-divisors on $Z$
as in the lemma.
Then $t_0P'-(K_Z+\D_Z-N')$ is big and nef and $\Nklt(Z,\D_Z-N')\cap \B(P')=\emptyset$.

Thus we can apply Theorem \ref{Ambro} and we are done.
\end{proof}

\begin{remark}\begin{em}\label{ambrocoro_bpf}
In Corollary  \ref{ambrocoro},
instead of requiring that $D$ is big and $aD-(K_X+\D)$ is nef, 
we can consider the condition that $aD-(K_X+\D)$ is big and nef for some $a \in \Q$.
In fact this is sufficient to apply Lemma \ref{allseason}.
\end{em}\end{remark}

If $X$ is $\Q$-Gorenstein, by using Theorem \ref{Ambro}, we get the following:

\begin{theorem}\label{Q-Gorenstein}
Let $(X,\D)$ be an LC pair such that $X$ is $\Q$-Gorenstein and $\D$ is effective.

Let $D \in \Div_\Q(X)$ be such that
\begin{enumerate}
\item $D$ is big;
\item $\B_+(D)\not\supseteq V$, for every $V \in \CLC(X,\D)$;
\item $D$ has a $\Q$-CKM Zariski decomposition $$D=P+N$$ with
 $\B(P)\cap V=\emptyset$ for every $V \in \CLC(X,\D)$ such that $V \not\subseteq\supp(\D)$;
\end{enumerate}

then there exists $\beta>0$ such that if
$$aD-(K_X+\D) \mbox{ is nef for some rational number } a>-\beta,$$

then $P$ is semiample.
\end{theorem}

\begin{proof}
Note that $P$ is big because $D$ is such and it is easy to see that $\B_+(P)=\B_+(D)$ and $\supp(N)\subseteq \B_+(D)$.
Then, thanks to Lemma \ref{giocoB_+}, we can find an effective Cartier divisor $\Gamma$ and a rational number $\la>0$ such that
$P-\la\Gamma$ is ample,  the pair $(X,\D+\la\Gamma)$ is LC and $\CLC(X,\D)=\CLC(X,\D+\la\Gamma)$.

Now, as $\supp(N)$ does not contain any LC center of the pair $(X,\D+\la\G)$, there exists $\beta\in \Q^+$ such that, if $0\leq\beta'<\beta$, then
the pair $(X,\D+\la\Gamma +\beta' N)$ is LC and $\CLC(X,\D+\la\Gamma +\beta' N)=\CLC(X,\D)$.

Suppose $a>-\beta$ is a rational number such that $aD-(K_X+\D)$ is nef.

Define $a':=-\min\{0,a\}$, $a'':=\max\{0,a\}$, so that $a=a''-a'$, $a''\geq 0$, $0\leq a'<\beta$.
Moreover we define $\D'=\D+\la\Gamma +a' N$, so that $\D'$ is effective, $(X,\D')$ is LC
and $\CLC(X,\D')=\CLC(X,\D)$.
Hence,  we get that for every $\varepsilon \in \Q^+$
$$\CLC(X,\D'-\varepsilon\D-a''N)\subseteq \CLC(X,\D'-\varepsilon \D)=$$
$$=\{V \in \CLC(X,\D) \mbox{ such that } V \not\subseteq \supp(\D)\},$$

so that, by hypothesis, $\B(P)$ does not intersect any LC center of the pair $(X,\D'-\varepsilon D-a''N)$.
Moreover
$$(1+a)P+a''N-(K_X+\D'-\varepsilon\D)=(1+a)P+a''N-(K_X+\D+\la\Gamma+a'N-\varepsilon\D)=$$
$$= (P-\la \Gamma)+ (aD-(K_X+\D))+\varepsilon\D$$

is ample if $\varepsilon$ is sufficiently small, thanks to the openness of the ample cone.

Thus we obtain the semiampleness of $P$ by applying Theorem \ref{Ambro} to the pair $(X,\D'-\varepsilon\D)$.
\end{proof}

\section{Alternative hypotheses}

In this section we state a version of Corollary \ref{overDLT} with more classical ``basepoint-free type" hypotheses
and we show that the proof is very similar.
Note that the same variation can be stated for Theorem \ref{Q-Gorenstein} and Corollary \ref{Nklt2}.

Moreover these ``basepoint-free type" hypotheses already appeared in Remark \ref{ambrocoro_bpf} and in Corollary \ref{DLT_bpf}.

\begin{corollary}
Let $(X,\D)$ be a pair such that $\D$ is effective.

Let $D \in \Div_\Q(X)$ be such that
\begin{enumerate}
\item $aD-(K_X+\D)$ is big and nef for some $a \in \Q$;
\item there exists a projective birational morphism $f:Z \to X$ such that $f^*(D)=P+N$ is a $\Q$-CKM Zariski decomposition and
\begin{itemize}
\item $(Z,\A(\D)_Z-aN)$ is an LC pair;
\item $\B_+(f^*(aD-(K_X+\D)))$ does not contain any LC center of the pair $(Z,\A(\D)_Z-aN)$;
\item $\Ndlt(Z,\A(\D)_Z-aN)=\emptyset$, or $P_{|_{\Ndlt(Z,\A(\D)_Z-aN)}}$ is semiample.
\end{itemize}
\end{enumerate}

Then $P$ is semiample.
\end{corollary}

\begin{proof}
Define $L:=f^*(aD-(K_X+\D))$.
Then we can apply Lemma \ref{giocoB_+} to the big and nef $\Q$-divisor $L$ and to the pair $(Z,\A(\D)_Z-aN)$ and we find a Cartier divisor $\Gamma$
and a rational number $\la>0$ such that
$L-\la\Gamma$ is ample, $(Z,\A(\D)_Z-aN+\la\Gamma)$ is LC and $\CLC(Z,\A(\D)_Z-aN+\la\Gamma)=\CLC(Z,\A(\D)_Z-aN)$.

Furthermore, we can choose $\Gamma$ generically in its linear series and we have that $Bs(|\Gamma|)=\B_+(L)$.
Then, by Bertini's Theorem, we can suppose that, outside $\B_+(L)$, $\G$ is smooth  and it intersects $\A(\D)_Z-aN$ in a simple normal crossing way.

Now we apply Lemma \ref{allseason} and we consider $t_0$, $D'$, $P'$, $N'$, $\D_Z$ as in the lemma.
Then $t_0P'-(K_Z+\D_Z+\la \G-N')=P+L-\la\G$
is ample.

We conclude by applying Theorem \ref{estate} to the pair $(Z,\D_Z+\la\G-N')=(Z,\A(\D)_Z-aN+\la\Gamma)$
and the $\Q$-Cartier divisor $P'$:

In fact and we can argue as in the proof of Corollary \ref{overDLT} to show that all the hypotheses of the theorem are verified.
\end{proof}
\section{Examples}\label{examples}
\subsection{Basic construction}
The following general construction is due to Hacon and McKernan (see \cite[Theorem A.6]{Laz09}).
The choice of the surface $S$ is due to Gongyo (see \cite[Example 5.2]{Gon12}).

$ $

Let $S$ be the surface obtained by blowing up $\mathbb{P}^2$ in $9$ very general points, so that $-K_S$ is nef but not semiample.
Let $S\subseteq \mathbb{P}^N$ be a projectively normal embedding.

Let $X_0$ be the cone over $S$ and let $\phi:X\rightarrow X_0$ be the blowing-up at the vertex.

We have that $X\simeq \mathbb{P}_S(\mathcal{O}_S\oplus \mathcal{O}_S(-H))$, where $H$ is a sufficiently ample divisor on $S$.
Now we denote by $\pi:X\rightarrow S$ the natural projection, and by $E$ the $\phi$-exceptional divisor,
so that $E\simeq S$.

Note that $-(K_X+E)$ is big and nef and $(X,E)$ is a PLT pair as $\CLC(X,E)=\{E\}$.

Moreover, by adjunction,
$-(K_X+E)_{|_E}=-K_E$,
whence $-(K_X+E)$ is not semiample because $-K_S$ is not semiample.

\subsection{Applications}

In Example \ref{es.1}
we will show that, with the notation of the previous subsection, $E\subseteq \B_+(-(K_X+E))$, but $E\not\subseteq \B(-(K_X+E))$.

Then we have that $(X,E)$ is a PLT (hence DLT) pair such that
\begin{enumerate}
\item $-(K_X+E)$ is big and nef;
\item $\B(-(K_X+E))$ does not contain the only LC center of the pair $(X,E)$;
\item $-(K_X+E)$ is not semiample, so that $R(X,-(K_X+E))$ is not finitely generated.
\end{enumerate}

$ $

In  Example \ref{es.2} we will construct, for every $k \in \N$, a $\Q$-divisor $P$ and a $\Q$-divisor $\D$ on $X$ such that
$(X,\D)$ is DLT and the following conditions are satisfied:

\begin{enumerate}
\item $P$ is big and nef;
\item $P-(K_X+\D)$ is big and nef;
\item The pair $(X,\D)$ has  $m\geq k$ LC centers and just one of these
is contained in $\B_+(P)$;
\item $P$ is not semiample, so that $R(X,P)$ is not finitely generated.
\end{enumerate}

Note that property (iii) implies
 that there is one LC center of $(X,\D)$, say $V$,
 such that $P$ remains big when restricted to every LC center in $\CLC(X,\D)\setminus\{V\}$.

$ $

These examples show that in Conjecture 1 and many of our theorems, e.g. Theorem \ref{DLT_BCHM}, Corollary \ref{overDLT} and Theorem \ref{Q-Gorenstein},
we cannot lighten the hypothesis on the $\B_+$, meaning that we cannot replace it with the same hypothesis on the stable base locus and
we must take into account \textit{all} the LC centers.

Similarly we cannot sharpen the hypothesis of logbigness of $P$ in Conjecture 2
as well as in Theorem \ref{DLT_logbig}, in Theorem \ref{estate2}
and in Corollary \ref{dim3}.

\begin{example}\begin{em}\label{es.1}
Note that $E \subseteq \B_+(-(K_X+E))$ because by Nakamaye's theorem we have that
$$\B_+(-(K_X+E))=\Null(-(K_X+E))$$
and $(-(K_X+E)^2\cdot E)=(-(K_X+E)_{|_E})^2=0$.

$ $

On the other hand we have that $E\not \subseteq \B(-(K_X+E))$:

In fact
$$h^0(E, -(K_X+E)_{|_E})=h^0(\mathbb{P}^2,\mathcal{I}_{\{p_1,...p_9\}}(3))\not=0.$$
Thus the surjectivity of the restriction map
$$H^0(X,-(K_X+E))\rightarrow H^0(E,-(K_X+E)_{|_E})\not=0,$$

given by  Kawamata-Viehweg vanishing theorem (\cite[Theorem 4.3.1]{Laz04}),
implies that $E\not\subseteq Bs(|-(K_X+E)|)$, so that in particular $E\not\subseteq \B(-(K_X+\nolinebreak E))$.
\end{em}\end{example}

\begin{example}\begin{em}\label{es.2}
Let $A_1,\dots, A_k$ be smooth hyperplane sections on $X_0$ such that $v \not\in A_i$ for every $i=1,\dots,k$ and
the ample divisor $A:=\sum A_i$ is SNC.
Let
$$P:=-(K_X+E)+\phi^*(A).$$
Moreover define $\D:=E+\phi^*(A)=E+\phi_*^{-1}(A)$.

Note that the pair $(X,\D)$ is DLT, because $X$ is smooth and $E+\phi_*^{-1}(A)$ is a SNC divisor,
and the LC centers of $(X,\D)$ are exactly the irreducible components of finite intersections of prime divisors
in the support of $\D$, namely $E$ and $\phi^*(A_i)$ for every $i \in\{1,\dots k\}$.
Note also that $P$ and $P-(K_X+\D)$ are big and nef.

Now we know that there exists $\varepsilon>0$ such that $\phi^*(A)-\varepsilon E$ is ample.
Then we can write $$P=\big{(}-(K_X+E)+\phi^*(A)-\varepsilon E\big{)} +\varepsilon E,$$

where $-(K_X+E)+\phi^*(A)-\varepsilon E$ is ample.

This implies that $\B_+(P)\subseteq E$.
On the other hand $\phi^*(A)\cap E=\emptyset$, so that the only LC center of the pair $(X,\D)$ contained in $\B_+(P)$ is $E$.

Moreover, as $\phi^*(A)_{|_E}=0$,
we have that $P_{|_E}=-(K_X+E)_{|_E}=-K_E$ is not semiample, because $E\simeq S$.
Therefore $P$ is not semiample.
\end{em}\end{example}



\begin{thebibliography}{99}
\bibliographystyle{alpha}

\bibitem[AHK07]{AHK07} V. Alexeev, C. Hacon, Y. Kawamata, \emph{Termination of (many) 4-dimensional log flips},
Invent. Math. 168 (2007), no.2, 433--448.
\bibitem[Amb01]{Amb01} Florin Ambro, \emph{Quasi-log varieties}, Tr. Mat. Inst. Steklova 240 (2003),
Biratsion. Geom. Linein. Sist. Konechno Porozhdennye Algebry, 220-239;
translation in Proc. Steklov Inst. Math. 2003, no. 1 (240), 214-233.
\bibitem[Amb05]{Amb05} Florin Ambro, \emph{A semiampleness criterion}, J. Math. Sci. Univ.Tokyo 12 (2005), 445-466.
\bibitem[BBP13]{BBP13} S. Boucksom, A. Broustet, G. Pacienza,
\emph{Uniruledness of stable base loci of adjoint linear systems with and without Mori theory},to appear in Math. Z.,
DOI 10.1007/s00209-013-1144-y
\bibitem[BCHM10]{BCHM10}C. Birkar, P. Cascini, C.~D. Hacon, and J. McKernan, \emph{Existence of minimal models for varieties of log general type},
 J. Amer. Math. Soc., 23(2):405--468, 2010.
\bibitem[BH13]{BH13} C. Birkar,  Z. Hu, \emph{Log canonical pairs with good augmented base loci}, arXiv: math.AG/1305.3569.
\bibitem[ELMNP06]{ELMNP06} L. Ein, R. Lazarsfeld, M. Mustat\u a, M. Nakamaye, M. Popa
\emph{Asymptotic invariants of base loci}. Ann. Inst. Fourier, Grenoble 56, p.1701,1734, 2006.
\bibitem[Fuj79]{Fuj79} Takao Fujita, \emph{On Zariski problem}, Proc. Japan Acad. Ser. A, 55 (1979), 106-110.
\bibitem[Fuj00]{Fuj00} Osamu Fujino, \emph{Abundance theorem for semi log canonical threeefolds}, Duke Math. J. 102 (2000), no. 3, 513-532.
\bibitem[Fuj07]{Fuj07} Osamu Fujino, \emph{What is log terminal?}, Flips for 3-folds and 4-folds, Oxford Lecture Series in Mathematics
and its Applications, vol.35, Oxford University Press, 2007, 49-62.
\bibitem[Fuj09]{Fuj09} Osamu Fujino, \emph{Introduction to the log minimal model program for log canonical pairs},
arXiv: math. AG/0907.1506v1.
\bibitem[Fuj11]{Fuj11} Osamu Fujino, \emph{Fundamental theorems for the log minimal model program},
 Publ. Res. Inst. Math. Sci. 47 (2011), no. 3, 727-789.
\bibitem[Fuj12]{Fuj12} Osamu Fujino, \emph{Basepoint-free theorems: saturation, b-divisors, and canonical formula},
Algebra Number Theory 6 (2012), no. 4, 797-823.
\bibitem[Gon12]{Gon12} Yoshinori Gongyo, \emph{On weak Fano varieties with log canonical singularities},
J. Reine. Angew. Math. 665 (2012), 237-252.
\bibitem[Kaw87]{Kaw87} Yujiro Kawamata, \emph{The Zariski decomposition of log-canonical divisors}.
Algebraic geometry, Bowdoin, 1985 (Brunswick, Maine, 1985), 425-433,
Proc. Sympos. Pure Math., 46, Part 1, Amer. Math. Soc., Providence, RI,
1987.
\bibitem[Kaw09]{Kaw09} Yujiro Kawamata, \emph{Finite generation of a canonical ring}. Current developments in mathematics, 2007, 43--76, Int. Press, Somerville, MA, 2009.
\bibitem[KM00]{KM00} J\'anos Koll\'ar, Shigefumi Mori, \emph{Birational geometry of algebraic varieties}. Cambridge Tracts in Mathematics, 134.
\bibitem[KMM85]{KMM85} Y. Kawamata, K. Matsuda, K. Matsuki, \emph{Introduction to the minimal model problem}, Algebraic geometry, Sendai, 1985,
283-360, Adv. Stud. Pure Math., 10, North-Holland, Amsterdam, 1987.

\bibitem[Kol05]{Kol05} J\'anos Koll\'ar, \emph{Resolution of Singularities} - Seattle Lecture, arXiv:math.
AG/0508332.
\bibitem[Laz04]{Laz04} Robert Lazarsfeld, \emph{Positivity in algebraic geometry, I-II}. Ergebnisse der Mathematik und ihrer
Grenzgebiete, 3. Folge, vol. 48-49, Springer-Verlag, Berlin, 2004.
\bibitem[Laz09]{Laz09} Vladimir Lazi\'c, \emph{Adjoint rings are finitely generated}, arXiv: math. AG/0905.2707v1.
\bibitem[Loh11]{Loh11} Daniel Lohmann, \emph{Families of canonically polarized manifolds over log Fano varieties}, arXiv: math. AG/1107.4545.
\bibitem[Mor87]{Mor87}Shigefumi Mori, \emph{Classification of higher-dimensional varieties}. Algebraic geometry, Bowdoin, 1985 (Brunswick, Maine, 1985),
269-332, Proc. Sympos. Pure Math., 46, Part 1, Amer. Math. Soc., Providence, RI, 1987.
\bibitem[Nak04]{Nak04} Noboru Nakayama, \emph{Zariski-decomposition and abundance}, MSJ Memoirs, vol. 14, Mathematical Society of Japan,
Tokyo, 2004.
\bibitem[Rei93]{Rei93} Miles Reid, \emph{Commentary by M. Reid}, (chapter 10 of Shokurov's paper ``3-fold log-flips"),
Russian Acad. Sci. Izv. Math. 40 (1993), 195-200.
\bibitem[Sza95]{Sza95} Endre Szab\'o, \emph{Divisorial log terminal singularities}. J. Math. Sci. Univ. Tokyo, 1:631-639, 1995.
\end{thebibliography}
\end{document}